\documentclass[a4paper]{article}

\usepackage[cp1250]{inputenc}
\usepackage{fancyhdr}
\usepackage{amssymb,amsmath,amsfonts,theorem,dsfont,a4wide}
\usepackage[T1]{fontenc}
\usepackage{epsfig}

\usepackage{color}
\usepackage[normalem]{ulem}

\usepackage{float}

\newtheorem{lemma}{Lemma}[section]

\newtheorem{corollary}{Corollary}[section]

\newtheorem{proposition}{Proposition}[section]

\newenvironment{proof}
   {\begin{trivlist}\item[]\textbf{\bf{Proof. }}\ignorespaces}
   {\qed\end{trivlist}}

\newcommand{\qed}{{\ifmmode q.e.d. \else\unskip\nobreak\hfil
\penalty50\quad\null\nobreak\hfill $\square$ \parfillskip=0pt
\finalhyphendemerits=0\par\fi}}

\begin{document}
\begin{center}
{\bf \Large Joins of 1-planar graphs}\\[3mm]
{\bf J\'ulius Czap}$^1$\\
{\bf D\'avid Hud\'ak}$^2$  \\
{\bf Tom\'a\v s Madaras}$^3$ \\
$^1$ Department of Applied Mathematics and Business Informatics, Faculty of Economics\\ Technical University of Ko\v{s}ice, N\v{e}mcovej 32, 040 01 Ko\v{s}ice, Slovakia\\
email: julius.czap@tuke.sk\\
$^2$ VSL Software, a.s., Lomen\'a 8, 040 01 Ko\v sice, Slovakia\\
email: davehudak@gmail.com\\
$^3$ Institute of Mathematics, Faculty of Science, Pavol Jozef \v Saf\'arik University \\Jesenn\'a 5, 040 01 Ko\v sice, Slovakia\\
email: tomas.madaras@upjs.sk
\end{center}

\bigskip
{\bf Abstract:} A graph is called 1-planar if there exists its drawing in the plane such that each edge is crossed at most once. In this paper, we study 1-planar graph joins. We prove that the join $G+H$ is 1-planar if and only if the pair $[G,H]$ is subgraph-majorized (that is, both $G$ and $H$ are subgraphs of graphs of the major pair) by one of pairs $[C_3 \cup C_3,C_3], [C_4,C_4], [C_4,C_3], [K_{2,1,1},P_3]$ in the case when both factors of the graph join have at least three vertices. If one factor has at most two vertices, then we give several necessary/sufficient conditions for the bigger factor. 



\section{Introduction}

Throughout this paper, we consider simple graphs (not necessarily connected); the used terminology is taken from \cite{D}. 

In research of graph properties (by graph property, we mean any family of graphs closed under isomorphism) one of typical questions concerns a characterization of graphs such that the result of a particular graph operation has the desired property. Among numerous results in this direction (see, for example, the classic paper \cite{BM} on various properties of Cartesian product) we highlight the ones concerning graph properties defined by topological or geometric constraints on drawings of graphs with respect to binary operation of graph join. Recall that the {\em join product} (or shortly, join) $G+H$ of two graphs $G$ and $H$ is obtained from vertex–disjoint copies of $G$ and $H$ by adding all edges between $V(G)$ and $V(H)$. In \cite{KM}, Kulli and Muddebihal  characterized planar joins; the crossing numbers of joins of special graphs were studied in \cite{K}, \cite{K2}, \cite{KS}, \cite{KS2} and \cite{china} (the crossing number $cr(G)$ of a graph $G$ is the minimum number of crossings in any plane drawing of $G$).

In this paper, we study the conditions under which the graph join is {\em 1-planar}, which means that there exists its drawing in the plane such that each edge is crossed at most once. The family of 1-planar graphs, although introduced already in 1965 by Ringel (see \cite{ringel}) in the connection with simultaneous vertex-face coloring of plane graphs, was explored in a deeper detail quite recently (see, for example, the papers \cite{FM}, \cite{fary}, \cite{mohar}); to our knowledge, a systematic study of graph operations that preserve 1-planarity was not realized.

Our results can be divided into two groups. In Section \ref{sec:3+}, we consider the cases that both factors of the graph join have at least three vertices; here, we obtain that the join $G+H$ is 1-planar if and only if the pair $[G,H]$ is subgraph-majorized (that is, both $G$ and $H$ are subgraphs of graphs of the major pair) by one of pairs $[C_3 \cup C_3,C_3], [C_4,C_4], [C_4,C_3], [K_{2,1,1},P_3]$. The cases when one of factors has one or two vertices are treated in Section \ref{sec:2-}; since the full characterization seems to be difficult, we give several necessary/sufficient conditions for the bigger factor.

When dealing with 1-planar graphs, we will use the following notation taken from \cite{FM}. Let $G$ be a 1-planar graph and let $D=D(G)$ be a 1-planar drawing of $G$ (that is, a drawing of $G$ in the plane in which every edge is crossed at most once; we will also assume that no edge is self-crossing and adjacent edges do not cross). Given two non-adjacent edges $pq$,  $rs$ $\in E(G)$, the \textit{crossing} of $pq$, $rs$ is the common point of two arcs $\stackrel{\frown}{pq}$, $\stackrel{\frown}{rs}$ $\in D$ (corresponding to edges $pq$, $rs$). Denote by $C=C(D)$ the set of all crossings in $D$ and by $E_0$ the set of all non-crossed edges in $D$. The \textit{associated plane graph} $D^{\times}=D^{\times}(G)$ of $D$ is the plane graph such that $V(D^{\times})=V(D)\cup C$ and $E(D^{\times})=E_0 \cup \{xz,yz|xy \in E(D)-E_0, z\in C,z\in xy\}$. Thus, in $D^{\times}$, the crossings of $D$ become new vertices of degree four; we call these vertices \textit{false}. Vertices of $D^{\times}$ which are also vertices of $D$ are called \textit{true}. Similarly, the edges and faces of $D^{\times}$ are called \textit{false}, if they are incident with a false vertex, and \textit{true} otherwise.

\section{The join $G+H$ with $|V(H)| \geq 3$}\label{sec:3+}
Let $C_n$ and $P_n$ denote the cycle and the path with $n$ vertices. The disjoint union of two graphs $G$ and $H$ is denoted by $G \cup H$. The disjoint union of $k$ identical graphs $G$ is denoted by $kG$.

Let us consider the join $G+H$ with $|V(G)|=m$ and $|V(H)|=n$ such that $m \geq n \geq 3$. Observe that the complete bipartite graph $K_{m,n}$ is a subgraph of $G+H$. In the following, we will use the characterization of 1-planar complete multipartite graphs from \cite{CH}; the results are contained in Table \ref{tab:1p}, where the notation of sizes of parts of vertices is the following: $a-b$ means the set $\langle a, b\rangle\cap\mathbb{Z}$ (the interval of integers); $a-$ means all integers greater or equal to $a$.

\begin{table}
\begin{center}
\begin{tabular}{|l|l|}
  \hline                                                  
  $k=2$ & $K_{1-,1}$~;~ $K_{2-,2}$~;~ $K_{3-6,3}$~;~ $K_{4,4}$  \\           
  $k=3$ & $K_{1-,1,1}$~;~ $K_{2-6,2,1}$~;~ $K_{2-4,2,2}$~;~ $K_{3,3,1}$  \\  
  $k=4$ & $K_{1-6,1,1,1}$~;~ $K_{2-3,2,1,1}$~;~ $K_{2,2,2,1-2}$  \\         
  $k=5$ & $K_{1-2,1-2,1,1,1}$   \\                         
  $k=6$ & $K_{1,1,1,1,1,1}$    \\                                                                                                     
  \hline
  \end{tabular}
  \caption{1-planar complete $k$-partite graphs.}
  \label{tab:1p}
  \end{center}
  \end{table}

These results imply that $G+H$ is not 1-planar if $m \geq 5, n \geq 4$ or $m \geq 7, n \geq 3$.

Assuming that $G+H$ is 1-planar, among all possible 1-planar drawings of $G+H$ we denote by $D$ such a drawing that has the minimum number of crossings. Let $c$ denote the number of crossings in $D$ and let $D^\times$ be the associated plane graph of $D$. Clearly, $D^\times$ has $|V(G)|+|V(H)|+c$ vertices and $|E(G)|+|E(H)|+|V(G)|\cdot|V(H)|+2c$ edges.
Observe that, for $m\ge n\ge3$, the join $G+H$ is 3-connected. This fact and the result of Fabrici and Madaras \cite{FM} (Lemma 2.1) imply that $D^\times$ is also 3-connected. 

\begin{lemma}\label{l:max2}
Let $G$ and $H$ be graphs such that $m=n=4$. If at least one of them contains a vertex of degree three, then  $G+H$ is not 1-planar.
\end{lemma}

\begin{proof}
If one of these graphs contains a vertex of degree three, then  $G+H$ cannot be 1-planar, since it contains a subgraph $K_{4,3,1}$ which is not 1-planar.
\end{proof}

\begin{lemma}\label{l:tri}
$(C_3 \cup P_1)+4P_1$ is not 1-planar.
\end{lemma}

\begin{proof}
Assume that $(C_3 \cup P_1)+4P_1$ is 1-planar. Then $D^\times$ has $8+c$ vertices and $19+2c$ edges. 

In \cite{KS2} it was proved that $cr((C_3 \cup P_1)+4P_1)=6$, hence $c \ge 6$.

Any plane triangulation on $8+c$ vertices contains $18+3c$ edges. 
No two false vertices are adjacent in $D^\times$, hence we can extend $D^\times$ to a triangulation by adding $18+3c - (19+2c) \ge 5$ edges which join only true vertices. No multiple edges arise, since otherwise the corresponding endvertices form a 2-vertex-cut in a 3-connected $D^\times$. 

First assume that the added edges join only the vertices of $4P_1$. Thus we have a 1-planar drawing of $(C_3 \cup P_1)+ K_{2,1,1}$. The graph $K_{2,1,1}$ contains a vertex of degree three, hence $(C_3 \cup P_1)+ K_{2,1,1}$ is not 1-planar (see Lemma \ref{l:max2}). Consequently, at least one added edge joins two vertices of $C_3 \cup P_1$. Similarly as above we have a 1-planar drawing of a non-1-planar graph ($C_3 \cup P_1$ with the added edge has a vertex of degree three), a contradiction.   
\end{proof}

\begin{lemma}
Let $G$ and $H$ be graphs such that $m=n= 4$. Then $G+H$ is 1-planar if and only if $G, H\subseteq C_4$.
\end{lemma}

\begin{proof}
If $G, H\subseteq C_4$, then $G+H$ is a subgraph of $K_{2,2,2,2}$, which is 1-planar. On the other hand, if $G+H$ is 1-planar, then the maximum degree of $G$ and $H$ is at most two (see Lemma \ref{l:max2}). Moreover, Lemma \ref{l:tri} implies that neither $G$ nor $H$ contain $C_3$ as a subgraph.  
\end{proof}

\begin{lemma}
Let $G$ and $H$ be graphs such that $m=n=3$. Then  $G+H$ is 1-planar.
\end{lemma}

\begin{proof}
The graph $G+H$ is a subgraph of $K_6$ which is 1-planar.
\end{proof}

\begin{figure}
\centerline{
\begin{tabular}{cccccc}
\includegraphics{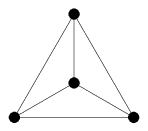}&
\includegraphics{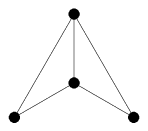}&
\includegraphics{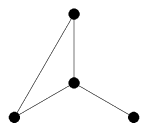}&
\includegraphics{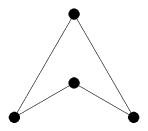}&
\includegraphics{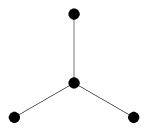}&
\includegraphics{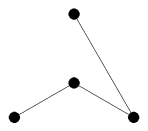}\\
$G_1$&$G_2$&$G_3$&$G_4$&$G_5$&$G_6$
\end{tabular}
}
\caption{Connected graphs on four vertices.}
\label{4}
\end{figure}

\begin{lemma}
Let $G$ and $H$ be graphs such that $m=4$ and $n= 3$. Then $G+H$ is not 1-planar if and only if  
\begin{enumerate}
\item $G=G_1$ and $H \subseteq C_3$, or
\item $G=G_i$ and $H=C_3$, $i=2,3,5$. 
\end{enumerate}
\end{lemma}

\begin{proof}
We distinguish several cases according to the number of edges of $G$.

{\bf Case 1.} If $G$ has six edges, then $G = G_1$. The join $G+H$ cannot be 1-planar because it contains the subgraph  $K_{3,1,1,1,1}$ which is not 1-planar.

{\bf Case 2.} If $G$ has five edges, then $G = G_2$. In this case $G+H$ is 1-planar if and only if $H\not= C_3$. If $H= C_3$, then $G+H=K_{2,1,1,1,1,1}$ which is not 1-planar. Otherwise $G+H$ is a subgraph of the 1-planar graph $K_{2,2,1,1,1}$.

{\bf Case 3.} If $G$ has four edges, then either $G=G_3$ or $G=G_4$. The graph $G_4+H$ is 1-planar, since it is a subgraph of the 1-planar graph $K_{2,2,1,1,1}$. The graph $G_3+C_3$ is not 1-planar, since it has a not 1-planar subgraph $K_{3,1,1,1,1}$. If $H\not =C_3$, then $H \subseteq K_{2,1}$. Consequently, $G_3+H$ is a subgraph of $K_{2,2,1,1,1}$ which is 1-planar. Therefore, the graph $G_3+H$ is 1-planar if and only if $H\not=C_3$.

{\bf Case 4.} If $G$ has at most three edges, then $G$ is a subgraph of $G_5$ or $G_6$ or $C_3 \cup P_1$. If $G= G_5$ and $H= C_3$, then $G+H = K_{3,1,1,1,1}$, so it is not 1-planar. Figure \ref{fig:figa} shows a 1-planar drawing of  $(C_3 \cup P_1)+C_3$, so  $(C_3 \cup P_1)+H$ is 1-planar for $H \subseteq C_3$. In the other cases $G+H$ is a subgraph of the 1-planar graph $K_{3,2,1,1}$ or $K_{2,2,1,1,1}$. 
\end{proof}

\begin{figure}
\centerline{
\includegraphics{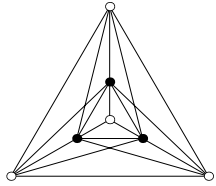}
}
\caption{A 1-planar drawing of  $(C_3 \cup P_1)+C_3$.}
\label{fig:figa}
\end{figure}

Now turn our attention to the case $m=5$ and $n=3$.

\begin{lemma}\label{l:g4+}
Let $G$ and $H$ be graphs such that $m=5$ and $n=3$. If $G$ contains a vertex of degree four, then $G+H$ is not 1-planar.
\end{lemma}

\begin{proof}
If $G$ has a vertex of degree four, then it contains $K_{4,1}$ as a subgraph. Consequently, $G+H$ contains the subgraph $K_{4,3,1}$ which is not 1-planar.
\end{proof}

\begin{lemma}\label{l:c3g}
Let $G$ be a graph on five vertices. If $G+C_3$ is 1-planar, then the maximum degree of $G$ is at most two.
\end{lemma}

\begin{proof}
If $G$ has a vertex of degree at least three, then it contains $K_{3,1}$ as a subgraph. Consequently, $G+C_3$ contains the subgraph $K_{3,1,1,1,1}$ which is not 1-planar. 
\end{proof}

To deal with the remaining possibilities, we color the edges and vertices of $G$ with blue, the edges and vertices of $H$ with red, and the other edges of $G+H$ (the edges which join vertices of $G$ with vertices of $H$) with black. Clearly, the black edges induce a complete bipartite graph $K_{5,3}$ whose crossing number is four, see \cite{Kleitman}. This means that, in any drawing of $G+H$, at least four crossings are incident only with black edges.
Let $e_1,e_2$ be two black edges which cross each other in $D$, and let $v$ be their crossing.

\begin{lemma}\label{l:p4p1c3}
$(P_4 \cup P_1)+C_3$ is not 1-planar.
\end{lemma}

\begin{proof}
Assume that $(P_4 \cup P_1)+C_3$ is 1-planar.
The red endvertices of $e_1,e_2$ and the false vertex $v$ form a false 3-face in $D^\times$ since otherwise we can decrease the number of crossings in $D$ or these red vertices would form a 2-vertex-cut in a 3-connected graph $D^\times$. 

There are three red edges in $D$ and at least four crossings are incident only with black edges, hence at least one red edge is incident with two false 3-faces in $D^\times$. Consequently, there are at least two blue vertices inside the red cycle $C_3$ and other two vertices outside the red one in $D$. There are five blue vertices, hence, without loss of generality, we can assume that there are exactly two blue vertices inside the red $C_3$. Therefore, there is only one crossing inside the red $C_3$. So, we obtain that the remaining  crossings (at least three), which are incident only with black edges, are outside the red $C_3$. Consequently, no red edge is crossed by any other edge in $D$. 

Two blue vertices are inside the red $C_3$, three blue vertices are outside the red $C_3$ and no red edge is crossed. Thus $(P_4 \cup P_1)+C_3$ cannot contain any blue $P_4$, a contradiction.  
\end{proof}

\begin{lemma}\label{l:p5p3}
$(P_4 \cup P_1)+P_3$ is not 1-planar.
\end{lemma}

\begin{proof}
Assume that $(P_4 \cup P_1)+P_3$ is 1-planar.
The red endvertex of $e_1$ is joined with the red endvertex of $e_2$ by a red edge in $D$ since otherwise, by adding this missing edge, we obtain a 1-planar drawing of the graph $(P_4 \cup P_1)+C_3$ which is not 1-planar. The red endvertices of $e_1,e_2$ and the false vertex $v$ form a false 3-face in $D^\times$.

There are two red edges in $D$ and at least four crossings are incident only with black edges, hence these red edges are incident with four false 3-faces in $D^\times$, see Figure \ref{fig:fig}. Consequently, the central vertex of the red $P_3$ is adjacent to at least four false vertices in $D^\times$.

\begin{figure}
\centerline{
\includegraphics{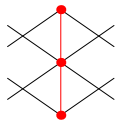}
}
\caption{The edges of $P_3$ are incident with four false 3-faces.}
\label{fig:fig}
\end{figure}

Now we extend $D^\times$ to a triangulation $T$ by adding some edges which join only true vertices. In this triangulation, the central vertex $x$ of the red $P_3$ has degree seven. The fact that $T$ is a triangulation implies that there is a 7-cycle in $T$ which contains all vertices adjacent to $x$. No false vertices are adjacent in $T$, hence this 7-cycle contains at most three false vertices. This means that $x$ is adjacent to at most three false vertices in $D^\times$, a contradiction.   
\end{proof}

\begin{lemma}\label{l:p5p2p1}
$(P_4 \cup P_1)+(P_2 \cup P_1)$ is not 1-planar.
\end{lemma}

\begin{proof}
Assume that $(P_4 \cup P_1)+(P_2 \cup P_1)$ is 1-planar.
The red endvertex of $e_1$ is joined with the red endvertex of $e_2$ by a red edge in $D$ since otherwise, by adding this missing edge, we obtain a 1-planar drawing of $(P_4 \cup P_1)+P_3$. The red endvertices of $e_1,e_2$ and the false vertex $v$ form a false 3-face in $D^\times$. For four such 3-faces, we need at least two red edges, but $(P_4 \cup P_1)+(P_2 \cup P_1)$ contains only one red edge, a contradiction.
\end{proof}

\begin{lemma}\label{l:p4p1p1p1p1}
$(P_4 \cup P_1)+3P_1$ is not 1-planar.
\end{lemma}

\begin{proof}
We can proceed similarly as in the proof of Lemma \ref{l:p5p2p1} using the fact that $(P_4 \cup P_1)+(P_2 \cup P_1)$ is not 1-planar.
\end{proof}

\begin{lemma}\label{l:k31p1p1}
$(K_{3,1} \cup P_1)+3P_1$ is not 1-planar.
\end{lemma}

\begin{proof}
Assume that $(K_{3,1} \cup P_1)+3P_1$ is 1-planar. Then the graph $D^\times$ has $8+c$ vertices and $18+2c$ edges. 

Any plane triangulation on $8+c$ vertices contains $18+3c$ edges, therefore $D^\times$ is not a triangulation. No two false vertices are adjacent in $D^\times$, hence we can extend $D^\times$ to a triangulation $T$ by adding $c\ge4$ (the graph $(K_{3,1} \cup P_1)+3P_1$ contains $K_{5,3}$ as a subgraph, hence $c \ge cr(K_{5,3})=4$) edges which join only true vertices. No multiple edges arise, since otherwise the corresponding endvertices form a 2-vertex-cut in $D^\times$. Clearly, at least one added edge joins two vertices of $K_{3,1} \cup P_1$. The graph $K_{3,1} \cup P_1$ with an extra edge has a vertex of degree four or it contains $P_4 \cup P_1$ as a subgraph, therefore the triangulation $T$ cannot be an associated plane graph of any 1-planar graph (see Lemmas \ref{l:g4+} and \ref{l:p4p1p1p1p1}), a contradiction. 
\end{proof}

\begin{lemma}\label{l:c3p2c3}
$(C_3 \cup P_2)+C_3$ is 1-planar.
\end{lemma}

\begin{proof}
$(C_3 \cup P_2)+C_3$ is a subgraph of $(C_3 \cup C_3)+C_3$. Figure \ref{fig:K36+} shows a 1-planar drawing of $(C_3 \cup C_3)+C_3$.
\end{proof}

\begin{figure}
\centerline{
\includegraphics[angle=90]{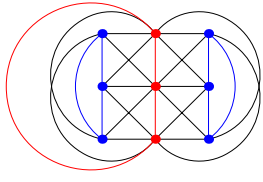}
}
\caption{A 1-planar drawing of $(C_3 \cup C_3)+C_3$.}
\label{fig:K36+}
\end{figure}

Finally, consider the cases with $m = 6$ and $n = 3$.

\begin{lemma}
If $G$ contains $P_4$ or $K_{3,1}$ as a subgraph, then $G+H$ is not 1-planar.
\end{lemma}

\begin{proof}
It follows from Lemmas \ref{l:p4p1p1p1p1} and \ref{l:k31p1p1}.
\end{proof}

\begin{lemma}
$(C_3 \cup C_3)+C_3$ is 1-planar.
\end{lemma}

\begin{proof}
Figure \ref{fig:K36+} shows a  1-planar drawing of $(C_3 \cup C_3)+C_3$.
\end{proof}

\begin{lemma}
$3P_2+3P_1$ is not 1-planar.
\end{lemma}

\begin{proof}
Assume that $3P_2+3P_1$ is 1-planar. 
$D^\times$ has $9+c$ vertices and $21+2c$ edges. 

$3P_2+3P_1$ contains $K_{6,3}$ as a subgraph, hence $c \ge cr(K_{6,3})=6$, see \cite{Kleitman}.

Any plane triangulation on $9+c$ vertices contains $21+3c$ edges, therefore $D^\times$ is not a triangulation ($21+2c<21+3c$).  
No two false vertices are adjacent in $D^\times$, hence we can extend $D^\times$ to a triangulation by adding $c\ge6$ edges which join only true vertices. No multiple edges arise, since otherwise the corresponding endvertices form a 2-vertex-cut in $D^\times$. By this way we obtain a 1-planar drawing of a supergraph of $(P_4 \cup P_2)+3P_1$, a contradiction.
\end{proof}

\section{The join $G+H$ with $|V(H)| \leq 2$}\label{sec:2-}

It is much more difficult to describe, in a "nice" way, for which graph $G$ is the join $G+P_1, G+2P_1$ or $G + P_2$ 1-planar. The cases with small $G$ are easy: if $|V(G)| \leq 4$, then the above mentioned joins are 1-planar (since they are subgraphs of $K_6$ which is 1-planar). In the general case with $G$ being larger, we start with easy necessary conditions:

\begin{lemma}
If $G+P_1$ is 1-planar, then $|E(G)|\le 3|V(G)|-4$.
\end{lemma}

\begin{proof}
From 1-planarity of $G+P_1$, we obtain $|E(G)|+|V(G)|=|E(G+P_1)|\le4|V(G+P_1)|-8=4|V(G)|-4$ which proves the claim. 
\end{proof}

Note that there are infinitely many positive integers $n$ for which there exist graphs $G_n$ on $n$ vertices and $3n-5$ edges such that $G_n+P_1$ is 1-planar: put $n = 2k$ with $k \geq 3$ being odd, take two paths $a_1 a_2\dots a_{k-1}a_k, b_1 b_{2} \dots b_{k-1} b_k$ and, for each $i \in \{1,\dots, k-1\}$, add new edges $a_ib_i, a_ib_{i+1},a_{i+1}b_i$ and the edge $a_kb_k$; in addition, for each odd $j \in \{1, k-2\}$, add new edges $a_ja_{j+2}$ and $b_jb_{j+2}$. The resulting graph $G_n$ is 1-planar with $n$ vertices and $3n-5$ edges and has an associated plane graph $G_n^\times$ such that, for each odd $j \in \{1, k-2\}$, the edges $a_ja_{j+2}, b_jb_{j+2}$ are true and incident with the outerface of $G_n^\times$ (see Figure \ref{fig:g10}). By putting a new vertex into this outerface and joining it with all vertices of $G_n^\times$, we obtain an 1-plane drawing of $G_n + P_1$.

\begin{figure}
\centerline{
\includegraphics{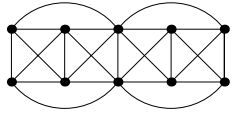}
}
\caption{The graph $G_{10}$.}
\label{fig:g10}
\end{figure}

\begin{lemma}\label{g2p1}
If $G+2P_1$ is 1-planar, then $|E(G)|\le 2|V(G)|$; moreover, this bound is sharp.
\end{lemma}

\begin{proof}
Again, from 1-planarity of $G+2P_1$, it follows that $|E(G)|+2|V(G)|=|E(G+2P_1)|\le4|V(G+2P_1)|-8=4|V(G)|$. To see that the bound is sharp, it is sufficient to consider the second power of the cycle $C_{2k}$, $k\ge3$.
\end{proof}
\begin{figure}
\centerline{
\includegraphics[angle=67.5]{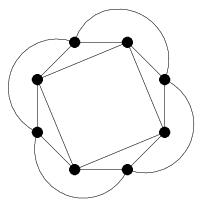}
}
\caption{The second power of the cycle $C_8$.}
\label{fig:cn2}
\end{figure}

In a similar manner, we can prove

\begin{lemma}
If $G+P_2$ is 1-planar, then $|E(G)|\le 2|V(G)|-1$. 
\end{lemma}

Note that for any positive integer $n\ge 5$ there exists a graph $G_n$ on $n$ vertices and $2n-2$ edges such that $G_n+P_2$ is 1-planar: Let $C_n=v_1v_2\dots v_nv_1$ be plane drawing of the cycle on $n$ vertices. Add the edges $v_iv_{i+2}$ for $i=1,3,\dots, 2\cdot\left\lfloor \frac{n-1}{2}\right\rfloor-1$ to the outer face of $C_n$ and add the edges $v_jv_{j+2}$ for $j=2,4,\dots, 2\cdot\left\lfloor \frac{n}{2}\right\rfloor-2$ to the inner face of $C_n$.

\begin{lemma}
If $G+2P_1$ is 1-planar, then $G$ does not contain $K_{7,1}$ or $K_{3,3}$ as a subgraph. 
\end{lemma}

\begin{proof}
If $G$ contains at least one of the graphs $K_{7,1}$, $K_{3,3}$ as a subgraph, then $G+2P_1$ contains a subgraph $K_{7,3}$ or $K_{3,3,2}$ which is not 1-planar.
\end{proof}

\begin{lemma}
If $G+P_2$ is 1-planar, then $G$ does not contain $K_{7,1}$, $K_{3,3}$, $K_{4,2}$, $K_{3,1,1}$ as a subgraph. 
\end{lemma}

\begin{proof}
If $G$ contains at least one of the graphs $K_{7,1}$, $K_{3,3}$, $K_{4,2}$, $K_{3,1,1}$ as a subgraph, then $G+P_2$ contains a non-1-planar subgraph from the set $\{K_{7,3}$, $K_{3,3,1,1}$, $K_{4,2,1,1}$, $K_{3,1,1,1,1}\}$.
\end{proof}

\begin{corollary}
If $G+2P_1$ or $G+P_2$ is 1-planar, then the maximum degree of $G$ is at most 6.
\end{corollary}

For sufficient conditions, we first deal with the case when $G$ is a plane graph. For this, we introduce several specialized notions:

A plane graph $G$ is called {\em 2-outerplanar} if the graph obtained from $G$ by deleting all vertices of the outerface of $G$ is outerplanar (the corresponding drawing is called 2-outerplane drawing).
A face of a plane graph $G$ is called {\em exposed} if it is adjacent to its outerface. Let ${\cal{P}}^\square$ be the family of all 2-outerplanar graphs with the property that, for each exposed face $f$, the number of vertices of $f$ which are not incident with the outerface is not greater than the number of common edges of $f$ and the outerface.

\begin{proposition}
\label{suf1}
If $G \in {\cal{P}}^\square$, then $G + P_1$ is 1-planar.
\end{proposition}
\begin{proof}
Taking a 2-outerplane drawing of $G \in {\cal{P}}^\square$, insert a new vertex $x$ into the outerface $f_O$ of $G$. Now, connect $x$ with all vertices incident with $f_O$ and, for each exposed face $f$, add new edges between $x$ and vertices which do not belong to $f_O$ in such a way that each new edge crossed exactly one edge incident with $f_O$; by the definition of ${\cal{P}}^\square$, this is always possible, hence, we obtain that the resulting graph, which is $G + P_1$, is indeed 1-planar. 
\end{proof}

Note that the converse of Proposition \ref{suf1} is not true: for graph $H$ in Figure \ref{fig:cube}, the join $H + P_1$ is 1-planar, but $H \not \in {\cal{P}}^\square$ (although $H$ is 2-outerplanar, for each its 2-outerplane drawing, the second condition of the definition of ${\cal{P}}^\square$ is not satisfied).

\begin{figure}
\centerline{
\includegraphics{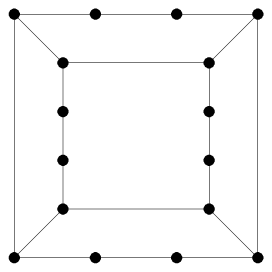}\hspace{1cm}
\includegraphics{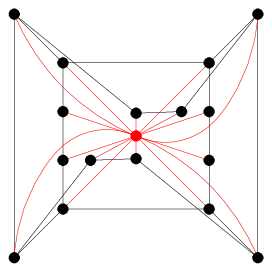}
}
\caption{The subdivided cube $H$ and the join $H + P_1$.}
\label{fig:cube}
\end{figure}

Another sufficient condition for 1-planarity of $G + P_1$ is based on the notion of so called {\em outer-1-planar graph}  defined by Eggleton in \cite{E} as a possible non-planar generalization of outerplanar graphs: a 1-planar graph $G$ is outer-1-planar if there exists its 1-planar drawing $D$ such that all true vertices of $D^\times$ lie on the outerface of $D^\times$. It is easy to see (arguing like in Proposition \ref{suf1}) that the following holds:

\begin{proposition}
\label{suf2}
If $G$ is outer-1-planar, then $G + P_1$ is 1-planar.
\end{proposition}

However, also the converse of Proposition \ref{suf2} is not true: Figure \ref{fig:out1} shows an example of 8-vertex graph which is not outer-1-planar (this follows from the fact that $n$-vertex outer-1-planar graph has at most $\frac{5}{2}n - 4$ edges, see \cite{Ax}) but yields 1-planar join with $P_1$.

\begin{figure}
\centerline{
\includegraphics{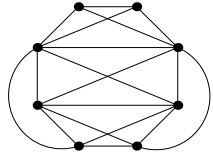}
}
\caption{A non outer-1-planar graph giving a 1-planar join with $P_1$.}
\label{fig:out1}
\end{figure}

All these examples show that the fact that $G + P_1$ is planar if and only if $G$ is outerplanar (see \cite{CLZ}) cannot be easily generalized to cover the 1-planarity of $G + P_1$ in terms of "nice" properties of $G$ (that is, relying only on the abstract structure of $G$).

\end{document}